\documentclass[12pt,reqno]{amsart}
\usepackage{mathtools,mathrsfs,cite,fullpage}
\usepackage{bbm}
\usepackage{stmaryrd}
\usepackage{textcomp}
\usepackage{setspace}
\usepackage{enumerate}
\usepackage{amssymb}
\usepackage{amsthm}
\usepackage{amsmath}
\usepackage{tikz}
\usetikzlibrary{arrows.meta, positioning}
\usepackage{marvosym}
\usepackage{empheq}
\usepackage{latexsym}
\usepackage[T1]{fontenc}
\usepackage[english]{babel}
\usepackage[utf8]{inputenc}
\usepackage{color}
\usepackage{hyperref}
\usepackage{cleveref}

\newtheorem{theorem}{Theorem}[section]
\newtheorem{lemma}[theorem]{Lemma}

\newtheorem{claim}[theorem]{Claim}
\newtheorem*{claim*}{Claim}

\theoremstyle{definition}

\newtheorem*{qu*}{Question}
\theoremstyle{remark}

\newcommand\E{\operatorname{\mathbb{E}}}

\renewcommand\leq{\leqslant}
\renewcommand\geq{\geqslant}

\renewcommand\ge{\geqslant}
\renewcommand\to{\rightarrow}

	\def\Prob{\mathbb{P}}

	\def\<{\langle }
	\def\>{\rangle }

\begin{document}

\title{A Rainbow version of Lehel's conjecture}
\author{ Pedro Ara\'{u}jo
  \and
  Xiao-Chuan Liu
  \and
  Ta\'{\i}sa Martins
  \and
  Walner Mendon\c{c}a
  \and
  Luiz Moreira
  \and
  Vinicius Fernandes dos Santos
  \and
  Zhifei Yan}

\address{(P. Araújo) Departamento de Matemática,
 Universidade Federal de Pernambuco,
	Avenida Jornalista Aníbal Fernandes - Cidade Universitária, Recife, Brasil}
\email{pedrocampos.araujo@ufpe.br}

\address{(X.C. Liu) Departamento de Matemática,
 Universidade Federal de Pernambuco,
	Avenida Jornalista Aníbal Fernandes - Cidade Universitária, Recife, Brasil}
\email{xiaochuan.liu@ufpe.br}

\address{(T. Martins) Instituto de Matemática, Universidade Federal Fluminense, Niterói, Brasil}
\email{tlmartins@id.uff.br}

\address{(W. Mendonça) Departamento de Matemática, Universidade Federal do Ceará, Fortaleza, Brasil}
\email{walner@mat.ufc.br}

\address{(L. Moreira) Departamento de Matemática,
 Universidade Federal de Pernambuco,
	Avenida Jornalista Aníbal Fernandes - Cidade Universitária, Recife, Brasil}
\email{luiz.fmoreira@ufpe.br}

\address{(V.F. dos Santos) Departamento de Ciência da Computação, Universidade Federal de Minas Gerais, Av. Antônio Carlos, 6627 - ICEX, Belo Horizonte, Brasil}
\email{viniciussantos@dcc.ufmg.br}

\address{(Z. Yan) ECOPRO, Institute for Basic Science, 55 Expo-ro, Yuseong-gu, Daejeon, 34126, Korea}
\email{zhifeiyan@ibs.re.kr}

\thanks{}

\begin{abstract}
Lehel's conjecture states that every $2$-edge-colouring of $K_n$ admits a partition of its vertex set into two monochromatic cycles.  It was proven for sufficiently large $n$ by  \L uczak, R\"odl, and Szemer\'edi in 1998,  later improved by Allen in 2008, and fully resolved by Bessy and Thomassé in 2010.

In this paper, we consider a rainbow analogue of Lehel's conjecture in the setting of properly edge-coloured complete graphs. We prove that, for sufficiently large $n$, every properly edge-coloured $K_n$ admits a partition of its vertex set into two vertex-disjoint rainbow cycles. 

\end{abstract}

\maketitle
\section{Introduction}
In 1979, Lehel~\cite{Ay79} conjectured that for every $2$-edge-colouring of the complete graph $K_{n}$ we can partition $V(K_{n})$ into two monochromatic cycles (here, a single vertex or a single edge is
considered a cycle). \L uczak, R\"odl, and Szemer\'edi~\cite{LRS98} proved in 1998 that Lehel's conjecture holds for sufficiently large $n$. Allen~\cite{Al08} later extended the result to smaller values of $n$ (although still sufficiently large).  
The conjecture was finally settled in 2010, when Bessy and Thomassé~\cite{BeTh10} presented an elementary and short proof that holds for every $n \geq 2$. Lehel's conjecture has motivated many other challenging problems in extremal combinatorics (see the survey~\cite{Gy16} for an exposition of some of these problems and the known results). A substantial number of these problems concern monochromatic structures in edge-colourings with a fixed number of colours. In this work, we are interested in the rainbow version of Lehel's conjecture.

A \emph{proper edge-colouring} of a graph $G$ is a colouring of $E(G)$ where edges of the same colour do not share vertices. Given a proper edge-colouring of $G$, we say that a subgraph $H \subseteq G$ is \emph{rainbow} if the colours of the edges of $H$ are all distinct. Our main result is the following.

\begin{theorem}\label{thm:main}
For every sufficiently large $n$ and any proper edge-colouring of $K_n$, there are two rainbow cycles partitioning $V(K_n)$.
\end{theorem}

It is not always possible to cover the entire $V(K_n)$ with a single rainbow path (and hence a rainbow Hamilton cycle). In fact, Maamoun and Meyniel~\cite{MaMe84} constructed a proper edge-colouring of $K_{n}$, for $n$ being a power of 2, with no rainbow Hamilton path. A famous conjecture of Andersen~\cite{A89} states that in every proper edge-colouring of the complete graph $K_n$ there is a rainbow path with $n-1$ vertices, which was in a recent breakthrough (independently of our work) proved by Bowtell, Montgomery, M{\"u}yesser and Pokrovskiy \cite{Bo26}. Such rainbow path together with the only uncovered vertex yields two rainbow paths partitioning $V(K_{n})$. In previous works, Alon, Pokrovskiy and Sudakov~\cite{APS17} had proved an approximate version of it by showing that there is a rainbow cycle of length $n - O(n^{3/4})$, which was later improved by Balogh and Molla~\cite{BM19} to $n-O(\sqrt{n}\log{n})$. 
In the proof of \Cref{thm:main}, we use this last result together with an absorber structure to cover the leftover vertices with another rainbow cycle.

The proof of \Cref{thm:main} follows the strategy below.
First we find an absorber $F$ which will consist of 
a fairly large robust subgraph that has the property that it can be covered with a single rainbow cycle, even if we add to it 'any' sufficiently small subset of vertices.
Then we cover most vertices in $V(K_n) \setminus F$ with a rainbow cycle $C_1$ using the results of Alon, Pokrovskiy, and Sudakov~\cite{APS17} or of Balogh and Molla~\cite{BM19}.
Let $X$ be the set of vertices in $V(K_n) \setminus F$ that are not covered by $C_1$.
Since $X$ is small enough to be absorbed by $F$, we finish the proof by picking the cycle $C_2$ that is guaranteed to cover $F \cup X$ by the absorption property of $F$. 

In our proof, we cannot guarantee that the cycles $C_1$ and $C_2$ do not share any colour. In fact, it is not always possible to partition $V(K_n)$ into two rainbow cycles that do not share any colour. If $n$ is even, there exists a proper edge-colouring of $K_n$ where each colour is a perfect matching.
In particular, such a colouring uses exactly $n-1$ colours and any partition of the vertices in two cycles will use $n$ edges. 

The remainder of the paper is organized as follows. In \Cref{sec:absorber}, we construct the absorber. In \Cref{sec:twocycles}, we employ this absorber to construct rainbow cycles partitioning $V(K_n)$, thereby establishing \Cref{thm:main}. We conclude with further remarks in \Cref{sec:discussion}.

\section{Building the absorber}\label{sec:absorber}

Our main aim in this section is to construct an absorber in any properly edge-coloured $K_n$ and to show that it can absorb any sufficiently small subset of vertices. 
To achieve this, in \Cref{sec:randomGraph} we first introduce a random subgraph $H \subseteq K_n$ and establish some of its probabilistic properties. Finally, in \Cref{sec:construction}, we build the absorber (see Figure~\ref{fig:absorber}) and demonstrate how it can be used to absorb 'any' sufficiently small subset of vertices.

\subsection{Random subgraph}\label{sec:randomGraph}
 In this subsection, we consider the following random subgraph, first introduced in \cite{APS17}, which plays a crucial role in the construction of our absorber.

Given a proper edge-colouring $\varphi:E(K_n) \to \mathbb{N}$ of the complete graph $K_n$, and a constant $p \in (0,1)$, we define the random graph $H = H(\varphi,p)$ as follows. Let $L \subseteq \varphi(E(G))$ be a random set obtained by choosing every colour in $\varphi(E(G))$ independently at random with probability $p$.
Then $H$ is the spanning subgraph of $K_n$ with edges 
$$E(H) = \{e \in E(K_n):\; \varphi(e) \in L\};$$
that is, the edges of $H$ are those edges of $K_n$ that are coloured by $\varphi$ with some colour in $L$. 

We now establish the properties of $H$.

\begin{lemma}\label{lem:regular}
Let $p \in (0,1)$ be constant. Then for every proper edge-colouring $\varphi$ of $K_n$, the following holds for $H = H(\varphi,p)$ with high probability.  
\begin{enumerate}[(i)]

\item For every $u \in V(H)$,
$$ pn/2 \leq |N_H(u)| \leq 2pn.$$

\item For every distinct $u,v \in V(H)$,
$$|N_H(u) \cap N_H(v)| \geq p^2n/6.$$

\item For every distinct $u,v,w \in V(H)$,
$$|N_H(u) \cap N_H(v) \cap
N_H(w)| \geq p^3n/14.$$

\item For any $\alpha \in (0,1)$ there exists a set $A \subseteq V(H)$ with $|A|= \lfloor \alpha n \rfloor$ such that for every $v \in V(H)$,
\begin{align}\label{eq:QQQ}
|N_H(v) \cap A| \leq 4p|A|.
\end{align}

\end{enumerate}

\end{lemma}

\begin{proof}
For (i), fix $u\in V(H)$. Since $\varphi$ is a proper edge colouring of $K_n$, notice that $|N_H(u)|$ is a binomial random variable with expectation $p(n-1)$. Therefore, by Chernoff's inequality,
\begin{align}\label{eq:N_H(u)}
\Prob(|N_H(u)| \notin [pn/2,2pn])\leq e^{-\Omega(pn)} = o(n^{-1}).    
\end{align}
The conclusion follows by taking a union bound over $u\in V(H)$.

For (ii), let $u,v\in V(H)$ be arbitrary vertices. Notice that for each vertex $x \in N_{K_n}(u) \cap N_{K_n}(v)$, there are at most 2 vertices $y \in N_{K_n}(u)\cap N_{K_n}(v)$ distinct from $x$ such that two of the edges $\{ux,vx,uy,vy\}$ have the same colour. Therefore, we can obtain a subset $N_{u,v} \subseteq N_{K_n}(u) \cap N_{K_n}(v)$ of size 
$$|N_{u,v}| \geq \frac{1}{3} |N_{K_n}(u) \cap N_{K_n}(v)| \geq \frac{n-2}{3}$$
such that for every pair of vertices $x$ and $y$ in $N_{u,v}$, all the edges in $\{ux,vx,uy,vy\}$ have distinct colours. 

Let $N'_{u,v}$ be the set of vertices $x\in N_{u,v}$ with $xu,xv\in E(H)$. Notice that $N'_{u,v} \subseteq N_H(u,v)$ and the definition of $N_{u,v}$ allows us to say that $|N'_{u,v}|$ is a binomial random variable with expectation $|N_{u,v}|p^2$. By Chernoff's inequality,
\begin{align}\label{eq:Nuv}
\Prob\left(|N'_{u,v}| \leq \frac{p^2n}{6}\right)\leq e^{-\Omega(p^2n)} = o(n^{-2}).    
\end{align}
A union bound over all the choices of $u$ and $v$, gives us that with high probability, we have $|N_H(u)\cap N_H(v)| \geq |N'_{u,v}| \geq p^2n/6$, for every distinct $u,v \in V(H)$.

For (iii), let $u,v,w\in V(H)$ be arbitrary distinct vertices. By a similar argument as above, we can create a set $N_{u,v,w}\subseteq V(H)\setminus\{u,v,w\}$ of size
$$|N_{u,v,w}|\geq (n-3)/7$$
such that all the edges, in the complete graph, between $\{u,v,w\}$ and any two vertices of $N_{u,v,w}$ have distinct colours. Then consider the set $N'_{u,v,w} = N_{u,v,w} \cap (N_H(u) \cap N_H(v) \cap N_H(w))$ and notice that $|N'_{u,v,w}|$ is a binomial random variable with expectation $|N_{u,v,w}|p^3$. Then, Chernoff's inequality gives
\begin{align}\label{eq:Nuvw}
\Prob\left(|N'_{u,v,w}| \leq \frac{p^3n}{14}\right)\leq e^{-\Omega(p^3n)} = o(n^{-3}).    
\end{align}
Finally, a union bound over all the choices of $u,v,w\in V(H)$ gives us that with high probability, we have $|N_H(u) \cap N_H(v) \cap N_H(w)| \geq |N'_{u,v,w}| \geq p^3n/14$, for every distinct vertices $u,v,w\in V(H)$.

For (iv), take a set $A \subseteq V(H)$ of size $|A|= \lfloor \alpha n \rfloor$.
Notice that $|N_H(v)\cap A|$ is a binomial random variable with expectation

$$\E\big[|N_H(v)\cap A|\big] = p|A|.$$
Hence by Chernoff's bound again, we have
\begin{align}\label{eq:Q2}
 \Prob(|N_H(v) \cap A|> 4p|A|)\leq e^{-\Omega(p\alpha n)} = o(n^{-1}).   
\end{align}
Taking the union bound over all the choices of $v$, we get that the probability that~\eqref{eq:QQQ} does not hold for some vertex $v$ is $o(1)$, which finishes the proof.

\end{proof}

\subsection{The absorber}\label{sec:construction}



The construction of the absorber begins with the set $A$ given by item $(iv)$ of Lemma \ref{lem:regular}. In this way, every vertex of $G=K_n$ is highly connected to $A$ by edges from $\overline{H}$, the complementary graph of $H$. This property of $A$ will be important in the second part of our absorbing strategy when we would like to cover any small set $X$ of vertices of $G$ that are left uncovered by an almost spanning rainbow cycle. Our first absorbing strategy, on the other hand, guarantees that we can adjoin to $A$ an auxiliary set of vertices $W$ with the property that $W \cup A$ is covered by a rainbow cycle in $H$. However, we will require some structure on $W$ which will help us to extend the rainbow cycle covering $W \cup A$ to a rainbow cycle covering $W \cup A \cup X$. 

The following lemma gives us the set $W$ that we will use to cover $A$. Notice that $W$ is formally a cycle, but we will also denote the set of vertices of such cycle by $W$ as well. We expect this notation to not cause any confusion.

\begin{lemma}
\label{lemma:absorber}
Let $p \in (0,1)$ be a constant and let $\varphi$ be a proper edge-colouring of $K_n$. Then the following holds with high probability for $H = H(\varphi,p)$.
For every $ A \subseteq V(H)$ with $|A|\leq 2^{-8}p^3n$, there exists $W\subseteq V(H)\setminus A$ with $|W|=2|A|+2$ and a rainbow path $P \subseteq H[W]$ of length $|A|$ such that for every $A' \subseteq A$ there is a rainbow cycle in $H$ with vertex set $W \cup A'$ which contains $P$ as a subgraph.
\end{lemma}

We say that the set $W$ obtained in the lemma above is an \emph{$A$-absorber in $H$}. We also say that the path $P$ is the \emph{spine} of $W$.

\begin{proof}
Fix an ordering of the vertices of
\(
A=\{a_1,a_2,\dots,a_t\}.
\)
We will show that with high probability there exists a rainbow subgraph $S$ in
$H$ with $V(S)=A \cup W$, in which $W =
\{v_0,v_1,\ldots,v_{2t+1}\}$ forms a cycle in $S$ and $v_0a_1v_1a_2v_2\cdots a_tv_t$ forms a path (see Figure \ref{fig:absorber}).

Once we have such a graph $S$, we can see that $W$ is an $A$-absorber in $H$ with spine $P = v_{t+1}v_{t+2}\ldots v_{2t+1}$. In fact, for every $A' \subseteq A$, we can obtain a rainbow cycle covering $W \cup A'$ by replacing the path $v_{i-1}v_{i}$ in $W$ with the path $v_{i-1}a_{i}v_{i}$, for each $i \in [t]$ such that $a_i \in A'$. Note that by doing this we keep the path $P$ as a subgraph of the cycle obtained and we also keep the cycle rainbow since all the edges in $S$ have distinct colours.

Let us now focus on showing that with high probability we can construct such graph $S$ in $H$.
Note that $S$ is a $3$-degenerate graph with the ordering of its vertices given by $a_1,\ldots,a_t,v_0,v_1,\ldots,v_{2t+1}$. For the sake of simplicity of the argument, we will artificially set $a_i=a_2$ for $i=t+1,t+2,\dots, 2t+1$.
We will construct $S$ in $H$ by first taking the vertices of $A$ as $a_1,\ldots,a_t$ and then greedily assigning the remaining vertices $v_0,v_1,\ldots,v_{2t+1}$ of $S$ while ensuring that all the respecting edges of $S$ between assigned vertices have distinct colours.

Suppose we have already chosen the vertices $v_0,\ldots,v_{i-1}$ of $S$ in $V(H)\setminus A$, for some $1 \leq i \leq 2t+1$, and let $E_{i-1}$ be the edges of $S$ in $A\cup \{v_0,v_1,\dots, v_{i-1}\}$. We will show that there exists a vertex $v_i \in V(H)\setminus (A \cup \{v_0,\ldots,v_{i-1}\})$ such that the edges $v_{i-1}v_i$, $a_iv_i$ and $v_ia_{i+1}$ have no colour in common with each other or with the colours of the edges in $E_{i-1}$.
By item (iii) in Lemma \ref{lem:regular}, with high probability there are at least $p^3n/14$ common neighbours of $v_{i-1}$, $a_i$ and $a_{i+1}$ in $H$. At most $|A|+i$ of those common neighbours are in $A \cup \{v_0,\ldots,v_{i-1}\}$. Since the colouring is proper, at most $3|E_{i-1}|$ of those common neighbours are such that their edges with $v_{i-1}$, $a_i$ or $a_{i+1}$ have a colour in common with the edges in the embedding so far. Therefore, there are at least$$ \frac{p^3n}{14} - (|A| + i) - 3|E_{i-1}| \geq \frac{p^3n}{14} - |A| - 4i -1 \geq \frac{p^3n}{14} - 5|A|$$
candidates for $v_i$.
Since $|A| \leq 2^{-8}p^3 n$ and $i < |A|$, it follows that there exists a vertex $v_i$ as desired. Note that we found a rainbow graph that contains $S$, since we artificially added some edges between $a_1$ and $v_i$ for $i=t+1,t+2,\dots, 2t+1$, which finishes the proof.
\end{proof}

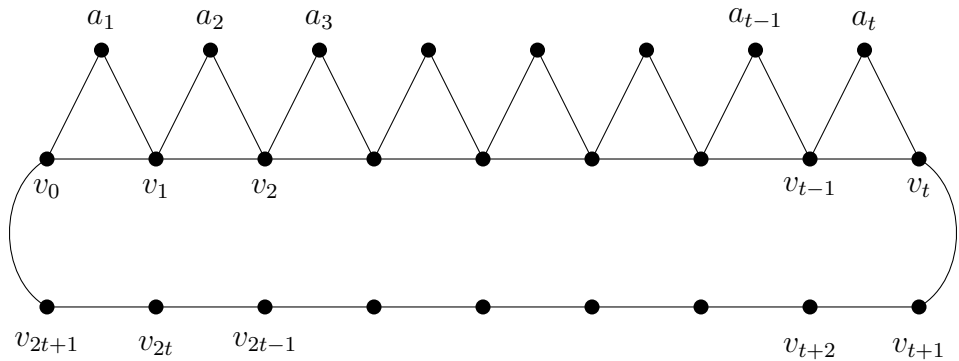
\begin{figure}[htbp]
\centering
\begin{tikzpicture}[scale=1.2, every node/.style={circle, fill=black, inner sep=2pt}]

  \def\dx{1.2}
  \def\h{3}
  \def\count{8}
  \pgfmathsetmacro{\countmenos}{\count - 1}

  \foreach \i in {0,...,\count} {
    \coordinate (u\i) at (\i*\dx, 0);
    \coordinate (l\i) at (\i*\dx, -1.6\h);
  }
  \foreach \i in {0,...,\countmenos} {
    \pgfmathsetmacro{\x}{(\i+0.5)*\dx}
    \coordinate (t\i) at (\x, 1.2);
  }

  \foreach \i in {0,...,\countmenos} {
    \draw (u\i) -- (u\the\numexpr\i+1\relax);
    \draw (l\i) -- (l\the\numexpr\i+1\relax);
    \draw (u\i) -- (t\i) -- (u\the\numexpr\i+1\relax);
  }
  \draw (u0) to[out=210,in=150] (l0);
  \draw (u\count) to[out=330,in=30] (l\count);

  \foreach \i in {0,...,\count} {
    \node at (u\i) {};
    \node at (l\i) {};
  }
  \foreach \i in {0,...,\countmenos} { \node at (t\i) {}; }

  \node[draw=none, fill=none, above=2pt] at (t0) {$a_{1}$};
  \node[draw=none, fill=none, above=2pt] at (t1) {$a_{2}$};
  \node[draw=none, fill=none, above=2pt] at (t2) {$a_{3}$};
  \node[draw=none, fill=none, above=-2pt] at (t6) {$a_{t-1}$};
  \node[draw=none, fill=none, above=2pt] at (t7) {$a_{t}$};
  \node[draw=none, fill=none, below=2pt] at (u0) {$v_{0}$};
  \node[draw=none, fill=none, below=2pt] at (u1) {$v_{1}$};
  \node[draw=none, fill=none, below=2pt] at (u2) {$v_{2}$};
  \node[draw=none, fill=none, below=-3pt] at (u7) {$v_{t-1}$};
  \node[draw=none, fill=none, below=2pt] at (u8) {$v_{t}$};
  \node[draw=none, fill=none, below=-2pt] at (l0) {$v_{2t+1}$};
  \node[draw=none, fill=none, below=4pt] at (l1) {$v_{2t}$};
  \node[draw=none, fill=none, below=-2pt] at (l2) {$v_{2t-1}$};
  \node[draw=none, fill=none, below=2pt] at (l7) {$v_{t+2}$};
  \node[draw=none, fill=none, below=2pt] at (l8) {$v_{t+1}$};
\end{tikzpicture}
\caption{Given a set $A = \{a_1,\ldots,a_t\}$ of vertices of $H$, we construct a rainbow cycle $W$ in $H$ with the vertices $\{v_0,\ldots,v_t,p_0,\ldots,p_t\}$ containing the rainbow path $P = (p_0,\ldots,p_t)$.}
\label{fig:absorber}
\end{figure}

We now describe how a small subset of vertices from 
$V(G)\setminus (A \cup V(W))$ can be absorbed. 
At this stage, the choice of the set $A$, in particular the fact that it satisfies property~(i) of Lemma~\ref{lem:regular}, plays a crucial role.

The absorption idea is to explore the fact that every vertex is highly connected to $A$ in colours outside $H$ to make $A$ absorb the vertices of $X$ connecting them to the edges of $P$ in $W$.

\begin{lemma}\label{lemma:absorption}
Let $p \in (0,1/16]$ be a constant and let $\varphi$ be a proper edge-colouring of $K_n$. Then the following holds with high probability for $H = H(\varphi,p)$. 
Suppose that $A\subseteq V(H)$ is such that $d_{H}(v,A) \leq 4p|A|$ for every $v\in V(H)$ and that $W$ is an $A$-absorber in $H$. Then for every $X \subseteq V(H) \setminus (A \cup W)$ with $|X| \leq |A|/40$, there exists a rainbow cycle (in $K_n$) that spans $X \cup A \cup W$.
\end{lemma}
\begin{proof}

Let $P=p_0p_1\cdots p_{|A|}$ be the path in $W$ given by~Lemma \ref{lemma:absorber}.
Let $X = \{x_1,x_2,\dots, x_\ell\}$, and set $C_0 = P$ and $A_0 = \emptyset$. 
We  greedily absorb each $ x_k \in X $ into the path $P$. Specifically, for each $ 1 \leq k \leq \ell $, assume we have constructed a rainbow path $ \mathcal{P}_{k-1} $ on the vertex set $ V(P) \cup A_{k-1} \cup \{x_1, \dots, x_{k-1}\} $ for some $A_{k-1} \subseteq A$. To absorb $ x_k $, 
we use the vertices from the edge $ p_{k-1}p_{k} $ 
together with two vertices from $A\setminus A_{k-1}$ to make a path from $p_{k-1} $ to $ p_{k} $ that passes through $x_k$.
More details are described below (see \Cref{fig:absortion}).

\begin{claim}\label{claim_of_sbsorption}
There exist two vertices $ a_{i_k}$ and $a_{j_k} \in A\setminus A_{k-1}$, such that the four edges
$$
x_k a_{i_k},\quad a_{i_k} p_{k-1},\quad x_k a_{j_k},\quad a_{j_k} p_{k}
$$
are all coloured distinctly and do not use colours in $\mathcal{P}_{k-1} $.
\end{claim}

\begin{proof}[Proof of Claim~\ref{claim_of_sbsorption}]

Notice that, by assumption, for every $v \in V(H)$,
$$
|N_{\overline{H}}(v) \cap A| \geq  |A| - 4p|A| \geq 3|A|/4,
$$
\noindent since $p\leq 1/16$. In particular, this implies that
\begin{align}\label{eq:tttt}
|N_{{\overline{H}}}(x_k)\cap N_{{\overline{H}}}(p_{k-1})\cap A|\geq  |A|/2.
\end{align}

Let us consider the common neighbours of $x_k$ and $p_{k-1}$ in ${\overline{H}}$.  By~\eqref{eq:tttt}, there are at least $|A|/2$ candidates for $a_{i_k}$. But in order to measure the valid candidates, we need to remove the vertices in $A_{k-1}$ and the vertices which at least one of the edges $x_ka_{i_k},\,a_{i_k}p_{k-1}$ have a previously used  colour. 
Notice that $|A_{k-1}|$ is at most $2|X|$. Moreover, since $P$ was constructed in $H$ and now we are in $\overline{H}$ and since the colouring is proper, the vertices that have a colour previously used in this process sum up to at most $8|X|$.
Thus, we need to avoid at most
$$2|X| + 8|X| = 10|X| \leq |A|/4$$ vertices from the $|A|/2$ candidates which ensures us that it is always possible to find an $a_{i_k}$.

By a similar argument, we consider the common neighbourhood of $x_k$ and $p_{k}$ and find an $a_{j_k}$.
\end{proof}

We define $A_k = A_{k-1} \cup \{a_{i_k}, a_{j_k}\}$ and $\mathcal{P}_k$ is obtained from $\mathcal{P}_{k-1}$ by deleting the edge $p_{k-1}p_k$ and adding the path $p_{k-1}a_{i_k}x_ka_{j_k}p_k$.
Notice that $\mathcal{P}_k$ is a rainbow path on the vertex set $ V(P) \cup A_{k} \cup \{x_1, \dots, x_{k}\} $ as desired.

At the end of the process, we obtain $A_\ell$ and $\mathcal{P}_\ell$.
To conclude the proof we consider the rainbow cycle $W'$ given by Lemma \ref{lemma:absorber} with $A' = A \setminus A_\ell$ and replace the path $P$ with $\mathcal{P}_\ell$.

\end{proof}

\begin{figure}[htbp]
\centering
\begin{tikzpicture}[scale=1.2, every node/.style={circle, fill=black, inner sep=2pt}]

  \def\dx{1.2}
  \def\h{3}
  \def\count{8}
  \pgfmathsetmacro{\countmenos}{\count - 1}

  \foreach \i in {0,...,\count} {
    \coordinate (u\i) at (\i*\dx, 0);
    \coordinate (l\i) at (\i*\dx, -1.6\h);
  }
  \foreach \i in {0,...,\countmenos} {
    \pgfmathsetmacro{\x}{(\i+0.5)*\dx}
    \coordinate (t\i) at (\x, 1.2);
  }
  \foreach \i in {1,...,3} {
    \pgfmathsetmacro{\x}{(\i+2)*\dx}
    \coordinate (x\i) at (\x, 3);
  }

  \foreach \i in {0,...,\countmenos} {
    \draw[black!20] (u\i) -- (u\the\numexpr\i+1\relax);
    \draw[black!20] (l\i) -- (l\the\numexpr\i+1\relax);
    \draw[black!20] (u\i) -- (t\i) -- (u\the\numexpr\i+1\relax);
  }
  \draw[black!20] (u0) to[out=210,in=150] (l0);
  \draw[black!20] (u\count) to[out=330,in=30] (l\count);

  \draw[thick] (l0) -- (t2) -- (x1) -- (t0) -- (l1);
  \draw[thick] (l1) -- (t3) -- (x2) -- (t6) -- (l2);
  \draw[thick] (l2) -- (t7) -- (x3) -- (t4) -- (l3);

  \foreach \i in {0,...,\count} {
    \node at (u\i) {};
    \node at (l\i) {};
  }
  \foreach \i in {0,...,\countmenos} { \node at (t\i) {}; }
  \foreach \i in {1,...,3} { \node at (x\i) {}; }

  \node[draw=none, fill=none, above=3pt] at (x1) {$x_1$};
  \node[draw=none, fill=none, above=3pt] at (x2) {$x_2$};
  \node[draw=none, fill=none, above=3pt] at (x3) {$x_3$};
  \node[draw=none, fill=none, below=2pt] at (l0) {$p_{0}$};
  \node[draw=none, fill=none, below=3pt] at (l1) {$p_{1}$};
  \node[draw=none, fill=none, below=3pt] at (l2) {$p_{2}$};
  \node[draw=none, fill=none, below=3pt] at (l3) {$p_{3}$};
  \node[draw=none, fill=none, above left=3pt] at (t0) {$a_{j_1}$};
  \node[draw=none, fill=none, above left=3pt] at (t2) {$a_{i_1}$};
  \node[draw=none, fill=none, above left=3pt] at (t3) {$a_{i_2}$};
  \node[draw=none, fill=none, above right=3pt] at (t6) {$a_{j_2}$};
  \node[draw=none, fill=none, above left=3pt] at (t4) {$a_{j_3}$};
  \node[draw=none, fill=none, above right=3pt] at (t7) {$a_{i_3}$};
\end{tikzpicture}
\caption{Given a set $X=\{x_1,x_2,\ldots\}$ of vertices of $G - \left(A \cup V(W)\right)$ relatively small, we construct a rainbow cycle that covers $X \cup A \cup V(W)$ by changing the rainbow cycle $W$ while taking \emph{detours} through $A$ instead of using the edges of the path $P$ contained in $W$.}
\label{fig:absortion}
\end{figure}
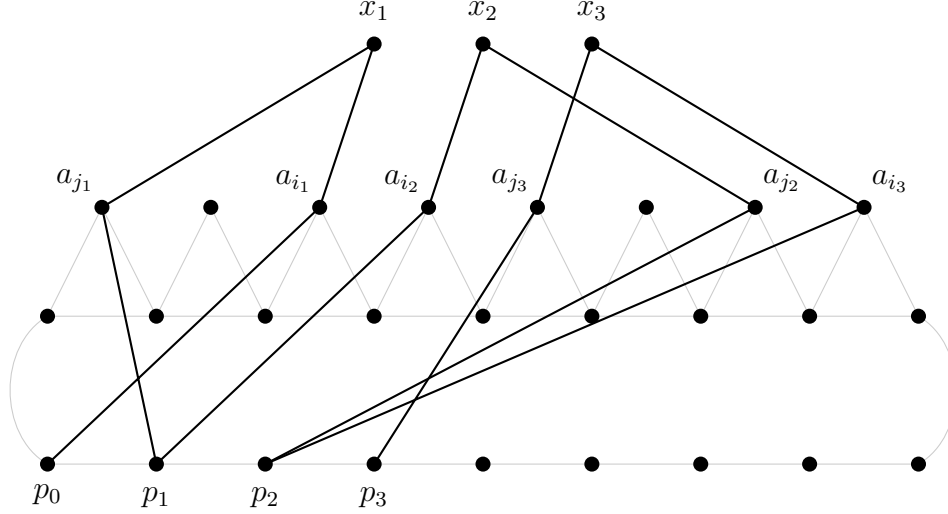

\section{Covering with two cycles}\label{sec:twocycles}

In this section, we will prove the following theorem which implies \Cref{thm:main}. We will use the following almost spanning result.

\begin{theorem}[\cite{BM19}]\label{thm:BM}
Let $\varphi$ be a proper edge-colouring of $K_n$ for a sufficiently large $n$. There exists a rainbow
cycle in $\varphi$ with at least $n - D\sqrt{n}\log n$ vertices for some constant $D>0$.
\end{theorem}

Now we are ready to prove the main theorem.


\begin{proof}[Proof of Theorem \ref{thm:main}]

Let $p \in (0,1/16]$ be any constant, let $n$ be sufficiently large and let $\varphi$ be a proper edge-colouring of $K_n$. Let $H$ be a typical sample of $H(\varphi,p)$ as in Lemma \ref{lem:regular} with $\alpha= 2^{-8}p^3$. Then there exists $A\subset V(H)$, with $|A|=\lfloor \alpha n\rfloor$ such that $|N_H(v)\cap A| \leq 4p|A|$ for every $v\in V(H)$. By Lemma \ref{lemma:absorber}, there exists $W\subset V(H)\setminus A$ which is an $A$-absorber in $H$. Since $|W|=2|A|+2$, we have that

\[n':=|V(H)\setminus (A\cup W)| = n - 3|A| - 2 \geq n -\dfrac{3p^3n}{2^8} -2\geq \dfrac{n}{2}.\]

By Theorem \ref{thm:BM}, there exists a rainbow cycle $C_1$ in $V(K_n)\setminus (A\cup W)$ that covers all but $D\sqrt{n'} \log n'$ vertices. Let $X$ be the vertices of $V(K_n)\setminus (A\cup W \cup V(C_1))$ and note that

\[|X| \leq D\sqrt{n'} \log n' = o(n).\]

Therefore, for sufficiently large $n$, we have that $|X|\leq |A|/40$ and by Lemma \ref{lemma:absorption} there exists a rainbow cycle $C_2$ that spans $X\cup A\cup W$. Therefore $C_1$ and $C_2$ are two vertex-disjoint rainbow cycles that partition the vertices of $K_n$, which finishes the proof.



\end{proof}



\medskip

\section{Concluding remarks}\label{sec:discussion}

In this paper, we established a rainbow analogue of Lehel’s conjecture by proving that every properly edge-coloured complete graph $K_n$ admits a partition of its vertex set into two vertex-disjoint rainbow cycles as long as $n$ is sufficiently large. Our proof combines the near-spanning rainbow cycle result of Balogh and Molla --- which guarantees a rainbow cycle covering all but $O(\sqrt{n}\log n)$ vertices in any proper edge-colouring of $K_n$ --- with a new absorbing structure that incorporates the remaining vertices into a second rainbow cycle disjoint from the first. As a consequence, the union of the two cycles spans $V(K_n)$.

As discussed in the introduction, it is not always possible to ensure that the two rainbow cycles are colour-disjoint. Our argument does not attempt to minimise the number of colours shared by the two cycles, and it would be interesting to understand to what extent such a refinement is possible.

A natural direction for future research is to investigate analogous decomposition problems into powers of cycles. For an integer $r \ge 2$, the \emph{$r$-th power} of a cycle $C_k$ (which we call an \emph{$r$-cycle}) is the graph on the same vertex set in which two vertices are adjacent whenever their distance along the cycle is at most $r$. By Vizing's Theorem, $K_n$ admits a proper edge-colouring with at most $n$ colours. In such colourings, any rainbow $r$-cycle can contain at most $\lfloor n/r \rfloor$ vertices, which makes it impossible to partition the vertices of $K_n$ into two rainbow $r$-cycles for any fixed $r$.

Nevertheless, it is natural to conjecture that every properly edge-coloured $K_n$ can be decomposed into a bounded number of vertex-disjoint rainbow $r$-cycles. A first step in this direction would be to prove that every properly coloured $K_n$ contains a rainbow $r$-cycle covering at least $n/r - o(n)$ vertices. Achieving such a result appears to require new ideas beyond the methods developed here.

\bigskip

\section*{Acknowledgements}

The main work of this paper was carried out during the $1^a$ Escola Brasileira de Combinatória held in São Sebastião. We are very grateful to Rob Morris for helpful suggestions on the presentation of this paper.


\bigskip
\bibliographystyle{alpha}
\addcontentsline{toc}{chapter}{Bibliography}
\bibliography{bibliography}

@article{BM19,
  title={Long rainbow cycles and Hamiltonian cycles using many colors in properly edge-colored complete graphs},
  author={Balogh, J. and Molla, T.},
  journal={European Journal of Combinatorics},
  volume={79},
  pages={140--151},
  year={2019},
  publisher={Elsevier}
}

@article{A89,
  title={Hamilton circuits with many colours in properly edge-coloured complete graphs},
  author={Andersen, L.},
  journal={Mathematica {S}candinavica},
  pages={5--14},
  year={1989},
  publisher={JSTOR}
}

@article{APS17,
  title={Random subgraphs of properly edge-coloured complete graphs and long rainbow cycles},
  author={Alon, N. and Pokrovskiy, A. and Sudakov, B.},
  journal={Israel Journal of Mathematics},
  volume={222},
  pages={317--331},
  year={2017},
  publisher={Springer},
  DOI={10.48550/arXiv.1608.07028}
}

@article {LRS98,
    AUTHOR = {{\L}uczak, T. and R\"{o}dl, V. and Szemer\'{e}di, E.},
     TITLE = {Partitioning two-coloured complete graphs into two
              monochromatic cycles},
  JOURNAL = {Combinatorics, Probability and Computing},
    VOLUME = {7},
      YEAR = {1998},
    NUMBER = {4},
     PAGES = {423--436},
      ISSN = {0963-5483},
   MRCLASS = {05C55 (05C38 05C80)},
  MRNUMBER = {1680072},
       DOI = {10.1017/S0963548398003599},
       URL = {https://doi.org/10.1017/S0963548398003599},
}

@article {Al08,
    AUTHOR = {Allen, P.},
     TITLE = {Covering two-edge-coloured complete graphs with two disjoint
              monochromatic cycles},
    JOURNAL = {Combinatorics, Probability and Computing},
    VOLUME = {17},
      YEAR = {2008},
    NUMBER = {4},
     PAGES = {471--486},
      ISSN = {0963-5483},
   MRCLASS = {05C70 (05C35)},
  MRNUMBER = {2433934},
       DOI = {10.1017/S0963548308009164},
       URL = {https://doi.org/10.1017/S0963548308009164},
}

@article {BeTh10,
    AUTHOR = {Bessy, S. and Thomass\'{e}, S.},
     TITLE = {Partitioning a graph into a cycle and an anticycle, a proof of
              {L}ehel's conjecture},
   JOURNAL = {J. Combin. Theory Ser. B},
  FJOURNAL = {Journal of Combinatorial Theory. Series B},
    VOLUME = {100},
      YEAR = {2010},
    NUMBER = {2},
     PAGES = {176--180},
      ISSN = {0095-8956},
   MRCLASS = {05C35 (05C38)},
  MRNUMBER = {2595702},
MRREVIEWER = {Elizabeth J. Billington},
       DOI = {10.1016/j.jctb.2009.07.001},
       URL = {https://doi.org/10.1016/j.jctb.2009.07.001},
}

@article{Gy16,
  title={Vertex covers by monochromatic pieces—a survey of results and problems},
  author={Gy{\'a}rf{\'a}s, A.},
  journal={Discrete Mathematics},
  volume={339},
  number={7},
  pages={1970--1977},
  year={2016},
  publisher={Elsevier}
}

@PhdThesis{Ay79,
  author   = {Ayel, J.},
  title    = {{Sur l'existence de deux cycles suppl{\'e}mentaires unicolores, disjoints et de couleurs diff{\'e}rentes dans un graphe complet bicolore}},
  school   = {{Universit{\'e} Joseph-Fourier - Grenoble I}},
  year     = {1979},
  type     = {Theses},
  month    = May,
  url      = {https://tel.archives-ouvertes.fr/tel-00289185},
}

@article{MaMe84,
  title={On a problem of {G}. {H}ahn about coloured Hamiltonian paths in $K_{2t}$},
  author={Maamoun, M. and Meyniel, H.},
  journal={Discrete mathematics},
  volume={51},
  number={2},
  pages={213--214},
  year={1984},
  publisher={Elsevier}
}

@article{Bo26,
  title={A proof of {A}ndersen's rainbow path conjecture for large $n$},
  author={Bowtell, C. and Montgomery, R. and M{\"u}yesser, A. and Pokrovskiy, A.},
  journal={arXiv:2608.06369},
  year={2026}
}

\end{document}